\documentclass[12pt]{amsart}
\usepackage{mathpazo}
\usepackage{hyperref}
\usepackage{bm}

\usepackage[margin=1in]{geometry}

\usepackage{graphicx}
\usepackage{tikz}

\usepackage{amssymb,amsmath,latexsym,amsthm,color}%or any usual package
\newtheorem{theorem}{Theorem}[section]

\newtheorem{proposition}[theorem]{Proposition}
\newtheorem{corollary}[theorem]{Corollary}
\newtheorem{predefinition}[theorem]{Definition}
\newenvironment{definition}{\begin{predefinition}\rm}{\end{predefinition}}
\newtheorem{preremark}[theorem]{Remark}
\newenvironment{remark}{\begin{preremark}\rm}{\end{preremark}}
\newtheorem{prenotation}[theorem]{Notation}
\newenvironment{notation}{\begin{prenotation}\rm}{\end{prenotation}}
\newtheorem{preexample}[theorem]{Example}
\newenvironment{example}{\begin{preexample}\rm}{\end{preexample}}
\newtheorem{preclaim}[theorem]{Claim}

\newtheorem{prequestion}[theorem]{Question}
\newenvironment{question}{\begin{prequestion}\rm}{\end{prequestion}}
\newtheorem{preapplication}[theorem]{Application}

\numberwithin{equation}{section}

\newcommand \ZZ {{\mathbb Z}}
\newcommand \QQ {{\mathbb Q}}
\newcommand \NN {{\mathbb N}}
\newcommand  \FF {{\mathbb F}}
\newcommand \GG {{\mathbb G}}
\newcommand \EE {{\mathbb E}}

\newcommand \dieu {{\mathbb D}}
\newcommand{\til}[1]{{\widetilde{#1}}}
\newcommand{\ang}[1]{\langle #1 \rangle}

\global\let\ker\undefined
\DeclareMathOperator{\ker}{Ker}
\DeclareMathOperator{\frob}{Fr}
\DeclareMathOperator{\ver}{Ver}
\DeclareMathOperator{\im}{Im}

\DeclareMathOperator{\bt}{BT}

\newcommand \CA {{\mathcal A}}
\newcommand \CM {{\mathcal M}}
\newcommand \CH {{\mathcal H}}

\begin{document}

\title[Newton polygons of curves]{Current results on Newton polygons of curves}

%first author

\author{\sc Rachel Pries}
\address{Rachel Pries\\
Colorado State University\\
Fort Collins, CO, 80521\\
}
\email{rachelpries@gmail.com}
\urladdr{http://www.math.colostate.edu/~pries/}


%11G Number theory - arithmetic algebraic geometry
%11G10 abelian varieties
%11G20 curves
%11M38 zeta and L-functions in char p

%14H Algebraic geometry - curves
%14H10 families, moduli
%14H40 Jacobians, Pryms

%12F05 (Field extensions) algebraic extensions

%14G17 (Arithmetic problems) positive char ground fields

%14K10 (Abelian varieties) algebraic moduli, classification
%14K15 (Abelian varieties) arithmetic ground fields

%14L05  (Algebraic groups) formal groups, pdiv groups

\maketitle
\begin{abstract}
There are open questions about which Newton polygons 
and Ekedahl-Oort types occur for Jacobians of smooth curves of genus $g$ in positive characteristic $p$.
In this chapter, I survey the current state of knowledge about these questions.
I include a new result, joint with Karemaker, which verifies, for each odd prime $p$, that there exist 
supersingular curves of genus $g$ defined over $\overline{\FF}_p$ for infinitely many new values of $g$. 
I sketch a proof of Faber and Van der Geer's theorem about the geometry of the $p$-rank stratification of the moduli space of curves.
The chapter ends with a new theorem, in which I prove that questions about the geometry of the Newton polygon and Ekedahl-Oort strata 
can be reduced to the case of $p$-rank $0$.
\end{abstract}

\smallskip

\thanks{
The author would like to thank the organizers of the conference 
Moduli Spaces and Arithmetic Geometry (Leiden, November 2015) and thank Achter, Karemaker, Oort, and the referee for helpful conversations.  The author was partially supported by NSF grant DMS-15-02227.}

\maketitle

 \mbox{}\\ \section{Introduction} \label{Sintro} \mbox{}\\

Suppose $X$ is a smooth projective connected curve 
defined over an algebraically closed field of prime characteristic $p$.
The Newton polygon and Ekedahl-Oort type are invariants of the Jacobian of $X$.
In general, it is not known which Newton polygons and Ekedahl-Oort types occur for Jacobians of smooth curves. 
More generally, it is not known how the open Torelli locus intersects the Newton polygon and Ekedahl-Oort strata in the moduli space ${\CA}_g$ of principally polarized abelian varieties of dimension $g$.  
In this paper, I survey the current state of knowledge and prove two new results about these open problems.

Section \ref{Ssupersingular} contains an introduction to the topic of supersingular curves.
Using the Newton polygon of the zeta function, we review what it means for a curve 
defined over a finite field to be supersingular and state the main open question about supersingular curves 
in Sections \ref{SssE} - \ref{Sopenss}.  
Sections \ref{Shermitian} - \ref{Sartsch} contain examples of supersingular Artin-Schreier curves from the literature,
including the construction in characteristic $p=2$ of supersingular curves of every genus by Van der Geer and Van der Vlugt.  
In Section \ref{Snewss}, Karemaker and I use their techniques to prove a new result 
which demonstrates, for each odd prime $p$, that there exist supersingular curves
of genus $g$ defined over $\overline{{\mathbb F}}_p$ for infinitely many new values of $g$, Corollary \ref{CKP}.

Section \ref{SNPEO} contains a survey about which Newton polygons and Ekedahl-Oort types are known 
to occur for Jacobians of smooth curves.  
In Sections \ref{Srevisit1}-\ref{Sqinv}, we give a basic definition of the $p$-rank, Newton polygon, and 
Ekedahl-Oort type and state the main open questions about these invariants for curves.
Section \ref{Smallg} contains results for small $g$.
Section \ref{Snonext} includes several non-existence results, 
which are almost all for curves with non-trivial automorphism group.

Section \ref{Snotation} contains precise definitions and notation, which are sometimes used in Section \ref{SNPEO}.
For a reader who is more familiar with curves than with abelian varieties, 
I suggest reading Section \ref{SNPEO} first
in order to gain intuition.
A reader who is familiar with abelian varieties in positive characteristic 
might prefer to read Section \ref{Snotation} first.

Starting in Section \ref{Sstratmoduli}, we study these open questions using moduli spaces.
Let ${\mathcal A}_g = {\mathcal A}_g \otimes \FF_p$ denote the moduli space of principally polarized abelian varieties of dimension $g$ in characteristic $p$ and ${\mathcal M}_g= {\mathcal M}_g \otimes \FF_p$ 
denote the moduli space of smooth curves of genus $g$ in characteristic $p$.
There are deep results about the stratifications of ${\mathcal A}_g$ by $p$-rank, Newton polygon, or Ekedahl Oort type;
however, there are very few results about how the open Torelli locus intersects these strata.
Sections \ref{Sdimstrata}-\ref{Sisogeny} include information about the dimensions of the strata in $\CA_g$.
In Section \ref{Sunlikely}, we describe how finding curves with an unusual Newton polyon can be viewed as an unlikely
intersection problem. 
In Sections \ref{SFVdG}-\ref{SFVdGproof}, I sketch a proof of \cite[Theorem 2.3]{fabervandergeer} by Faber and Van der Geer,
which states that every component of the $p$-rank $f$ locus of ${\mathcal M}_g$ has dimension $2g-3+f$,
for all primes $p$, and integers $g \geq 2$ and $f$ such that $0 \leq f \leq g$.

The main result of the paper is in Section \ref{Sinductive}, where I prove Theorem \ref{Tinductive} 
which shows that questions about the geometry of the 
Newton polygon and Ekedahl-Oort strata can be reduced to the case of $p$-rank $0$.
This is an inductive result, similar in spirit to earlier results in the literature,
but which allows for more flexibility with the Newton polygon and Ekedahl-Oort type.
As an application, for all primes $p$ and all $g \geq 4$, 
we show that there exists a smooth curve 
over $\overline{\mathbb{\FF}}_p$ of genus $g$ and $p$-rank $g-4$ whose Newton polygon has slopes $0, 1/4, 3/4, 1$, Corollary \ref{Amoduli}.

Section \ref{Sopen} contains several open problems.

\mbox{}\\ \section{Supersingular curves} \label{Ssupersingular} \mbox{}\\

This section contains an introduction to the topic of supersingular curves defined over finite fields, beginning with the 
case of supersingular elliptic curves.  We give an overview of some results in the literature about supersingular 
Artin-Schreier curves.
Corollary \ref{CKP} is a new result, joint with Karemaker, 
which determines, for each odd prime $p$, infinitely many new values of $g$ for which there exists a supersingular curve of 
genus $g$ defined over $\overline{\FF}_p$. 

In this section, $\FF_q$ is a finite field of cardinality $q=p^r$ and 
$\overline{\FF}_p$ denotes a fixed algebraic closure of $\FF_q$.
Let $X$ be a smooth projective connected curve of genus $g$ defined over $\FF_q$.

\mbox{}\\ \subsection{Supersingular elliptic curves} \label{SssE} \mbox{}\\

We first consider the case when the genus is $g=1$.
Let $E/\FF_q$ be an elliptic curve.
Let $a \in \ZZ$ be such that $\#E(\FF_q)=q+1-a$.
The zeta function of $E/\FF_q$ is 
\[Z(E/\FF_q, T)= \frac{1-aT+qT^2}{(1-T)(1-qT)}.\]

The supersingular condition was studied by Deuring \cite{deuring}.
As seen in \cite[Theorem V.3.1]{aec}, there are many equivalent ways to define what it means for $E$ to be supersingular.
In this section, we say $E/\FF_q$ is supersingular when $p \mid a$, see \cite[page 142]{aec}; 
otherwise $E$ is ordinary.

If $p=2$, then $E:y^2+y=x^3$ is supersingular, see Lemma \ref{Lhermitian}. 
In fact, this is an equation for the unique isomorphism class of 
supersingular elliptic curve over $\overline{\FF}_2$.

By \cite[Example V.4.4]{aec}, the elliptic curve $E:y^2=x^3+1$ ($j$-invariant $0$) is supersingular if and only if 
$p \equiv 2 \bmod 3$ and $p$ is odd.
By \cite[Example V.4.5]{aec}, the elliptic curve $E:y^2=x^3+x$ ($j$-invariant $1728$)
is supersingular if and only if $p \equiv 3 \bmod 4$. 
When $p=3$, this is an equation for the unique isomorphism class of 
supersingular elliptic curve over $\overline{\FF}_3$.

Every supersingular elliptic curve defined over a finite field of characteristic $p$ is, in fact, defined over $\FF_{p^2}$.
As seen in \cite[Theorem V.4.1]{aec}, for $p$ odd, Igusa proved that
\[E_\lambda: y^2=x(x-1)(x-\lambda)\] is supersingular for exactly $(p-1)/2$ choices of 
$\lambda \in \overline{\FF}_p$;
this shows that the number of isomorphism classes of supersingular elliptic curves
is $\lfloor \frac{p}{12} \rfloor + \epsilon$ with $\epsilon = 0,1,1,2$ when $p \equiv 1,5,7,11 \bmod 12$ respectively.

\mbox{}\\ \subsection{The Newton polygon of a curve} \label{SNPcurve} \mbox{}\\

Let $X/\FF_q$ be a smooth projective curve of genus $g$ and let ${\rm Jac}(X)$ denote its Jacobian.
Let $N_s = \#X(\FF_{q^s})$ be the number of points of $X$ defined over $\FF_{q^s}$.
The zeta function of $X/\FF_q$ is 
\[Z(X/\FF_q, T)={\rm exp}(\sum_{s=1}^\infty \frac{N_s T^s}{s}).\]

By the Weil conjectures for curves \cite[\S IV, 22]{weil}, \cite[\S IX, 69]{weil2},
there is a polynomial $L(X/\FF_q, T) \in \ZZ[T]$ of degree $2g$ such that 
\[Z(X/\FF_q, T)=\frac{L(X/\FF_q, T)}{(1-T)(1-qT)}.\]
The characteristic polynomial of the Frobenius endomorphism of the Jacobian of $X$ is 
$P(\mathrm{Jac}(X)/\FF_{q},T) = T^{2g}L(X/\FF_{q},T^{-1})$.
Furthermore, \[L(X/\FF_q, T)= \prod_{i=1}^{2g} (1- \alpha_i T),\] where 
the reciprocal roots $\alpha_i$ have the property that $|\alpha_i| = \sqrt{q}$.
This means that the complex roots of $L(X/\FF_q, T)$ all have archimedean absolute value $1/\sqrt{q}$.

The Newton polygon is a combinatorial object that keeps track of the $p$-adic valuations of the roots or, equivalently, 
of the coefficients
of $L(X/\FF_q, T)$.  Let $v_i$ be the $p$-adic valuation of the coefficient of $T^i$ in $L(X/\FF_q, T)$.
Let $v_i/r$ be its normalization for the extension $\FF_q/\FF_p$, where $q=p^r$.
The Newton polygon is the lower convex hull of the points $(i, v_i/r)$ for $0 \leq i \leq 2g$.

The Newton polygon consists of finitely many line segments, which break at points with integer coefficients, 
starting at $(0,0)$ and ending at $(2g,g)$.
If the slope $\lambda$ appears with multiplicity $m$, then so does the slope $1-\lambda$.

\mbox{}\\ \subsection{An open question about supersingular curves} \label{Sopenss} \mbox{}\\

\begin{definition}
The curve $X/\FF_q$ is {\it supersingular} if the Newton polygon of $L(X/\FF_q, T)$ is a line segment of slope $1/2$.  
\end{definition}

There are several other ways to characterize the supersingular property.
A curve $X/\FF_q$ of genus $g$ is supersingular if and only if:
\begin{enumerate}
\item the normalized Weil numbers $\alpha_i/\sqrt{q}$ are all roots of unity \cite[Theorem 4.1]{maninthesis}.
\item ${\rm End}_{\overline{\FF}_q}({\rm Jac}(X)) \otimes \QQ \simeq M_g(D_p)$, 
where $D_p$ is the quaternion algebra ramified only over $p$ and $\infty$ \cite[Theorem 2d]{tate:endo}.
%\item the formal group of ${\rm Jac}(X)$ being geometrically isogenous to $G_{1,1}^g$ \cite[Section 1.4]{LO}.
\item ${\rm Jac}(X)$ is geometrically isogenous to $E^g$ for some 
supersingular elliptic curve $E/\overline{\FF}_p$  \cite[Theorem~4.2]{O:sub}, which relies on \cite[Theorem 2d]{tate:endo}.
\end{enumerate}

Here is the main open question about supersingular curves:

\begin{question} \label{Qopenss}
Given a prime $p$ and $g \in \NN$, does there exist 
a supersingular smooth connected projective curve $X$ of genus $g$ defined over a finite field of characteristic $p$?
\end{question}

When $p = 2$, the answer to Question \ref{Qopenss} is yes for all $g \in \NN$, see Theorem \ref{Tasss2}.
For a fixed odd prime $p$, the answer is yes for infinitely many $g \in \NN$, see Proposition \ref{Lhermitian}, Theorem \ref{Tasss},   
and Corollary \ref{CKP}.
In Section \ref{Smallg}, we explain why the answer is yes for all $p$ when $g=1,2,3$.
The first open situation for Question \ref{Qopenss} is when $g=4$, 
in which case the answer is known to be yes only when $p$ satisfies certain congruence conditions.

\mbox{}\\ \subsection{Hermitian curves are supersingular} \label{Shermitian} \mbox{}\\

The Hermitian curve $H_q$ is the curve in ${\mathbb P}^2$ defined by the affine equation $y^q+y=x^{q+1}$.
Because $H_q$ is a smooth plane curve, the genus of $H_q$ is $g=q(q-1)/2$.

\begin{proposition} \label{Lhermitian} \cite[VI 4.4]{sti09}, \cite[Proposition 3.3]{HansenDL}
The Hermitian curve $H_q$ is maximal over $\FF_{q^2}$.  Also $L(H_q/\FF_q, T)=(1+qT^2)^g$ and $H_q$ is supersingular.
\end{proposition}

\begin{proof}
By definition, the curve $H_q$ is maximal over $\FF_{q^2}$ 
when the number of points realizes the upper bound in the Hasse-Weil bound.  So it suffices to show
\[\#H_q(\FF_{q^2})= q^2+1 + 2g q = q^3+1,\]
where the last equality uses that $g=q(q-1)/2$.  Since $H_q$ has one point at infinity, 
it suffices to show that there are $q^3$ pairs $(x,y) \in (\FF_{q^2})^2$ such that $y^q+y=x^{q+1}$.
Let ${\rm Tr}: \FF_{q^2} \to \FF_q$ and $N: \FF_{q^2}^* \to \FF_q^*$ denote the trace and norm maps respectively.
If $y \in \FF_{q^2}$, then ${\rm Tr}(y)=y^q+y \in \FF_q$; also ${\rm Tr}$ is a surjective $q$-to-$1$ homomorphism.
If $x \in \FF_{q^2}^*$, then $N(x) = x^{q+1} \in \FF_q$; also $N$ is a surjective $(q+1)$-to-$1$ homomorphism.
For $x=0$, there are $q$ values of $y \in \FF_{q^2}$ such that $y^q+y=0$.
Given $z \in \FF_q^*$, the number of pairs $(x,y) \in \FF_{q^2}^2$ such that $z=y^q+y=x^{q+1}$ is $q(q+1)$.
This yields $q + q(q-1)(q+1)=q^3$ affine points, which completes the proof of the first statement.

Since $H_q/\FF_{q^2}$ is maximal, the reciprocal roots $\alpha_i$ of $L(H_q/\FF_{q^2}, T)$ all equal $-\sqrt{q}$.
This implies that the reciprocal roots of $L(H_q/\FF_q, T)$ are $\pm \sqrt{-q}$ with multiplicity $g$.
It follows that $L(H_q/\FF_q, T)=(1+qT^2)^g$.  
Since $q=p^r$, for $0 \leq i \leq 2g$, one can check with the binomial theorem that $v_i/r \geq i/2$, where
$v_i$ is the $p$-adic valuation of the coefficient of $T^i$ in $L(H_q/\FF_q, T)$.
Thus $H_q$ is supersingular.
\end{proof}

\mbox{}\\ \subsection{Supersingular Artin-Schreier curves} \label{Sartsch} \mbox{}\\

To begin, given a prime $p$, consider the following curve:
\[Y_0: y^2-y=x^3 {\rm \ when \ } p=2 {\rm \ and \ } Y_0: y^p-y=x^2 {\rm \  when \ } p {\rm \  is \ odd.}\]
When $p=2$, $Y_0$ is a supersingular elliptic curve by Proposition \ref{Lhermitian}.  When $p$ is odd, 
$Y_0$ is a hyperelliptic curve with genus $(p-1)/2$ by the Riemann-Hurwitz formula.
% and $p$-rank $0$ by the Deuring-Shafarevich formula.  
In fact, $Y_0$ is supersingular when $p$ is odd.
This is due to work of Hasse and Davenport \cite{hassedavenport}
who proved that the eigenvalues of Frobenius for $Y_0$ are Gauss sums and so the characteristic polynomial of Frobenius is a $(p-1)$st power of a linear polynomial.
An alternative way to prove that $Y_0$ is supersingular is to note that it is a quotient of the Hermitian curve $H_{p^2}$.
 
A polynomial $R(x) \in k[x]$ is additive if $R(x_1+x_2)=R(x_1)+R(x_2)$.  This is equivalent to $R(x)$ being a 
linear combination of monomials $x^d$ where $d$ is a power of $p$.
Suppose that $R(x)$ is additive and consider the Artin-Schreier curve
\[Y:y^p-y=xR(x).\]
Then one can show that ${\rm Jac}(Y)$ is isogenous to a product of Jacobians
of Artin-Schreier curves arising from additive polynomials of smaller degree.
Working inductively, this shows that ${\rm Jac}(Y)$ is isogenous to a product of supersingular curves (isomorphic to $Y_0$).
This yields the following result.

\begin{theorem} \label{Tasss} \cite[Theorem 13.7]{VdGVdV92}, \cite[Corollary 3.7(ii)]{blache}, 
\cite[Proposition 1.8.5]{bouwzeta}
If $\FF_q$ is a finite field of characteristic $p$ and
$R(x) \in \FF_q[x]$ is an additive polynomial of degree $p^h$, then $Y:y^p-y=xR(x)$ is supersingular with genus $p^h(p-1)/2$.
\end{theorem}

More generally, the $L$-polynomial of $Y$ is determined in \cite[Theorem 1.8.4]{bouwzeta}.

In \cite{VdGVdV}, when $p=2$, the authors use Theorem \ref{Tasss} to prove that there exists a supersingular curve of every genus.

\begin{theorem}\label{Tasss2} \cite[Theorem 2.1]{VdGVdV}
If $p=2$ and $g \in \NN$, then there exists a supersingular curve $Y_g$ of genus $g$ defined over 
a finite field of characteristic $2$.
\end{theorem}

In fact, the authors prove that $Y_g$ can be defined over $\FF_2$ \cite[Theorem 3.5]{VdGVdV}.

\mbox{}\\ \subsection{New result about supersingular curves} \label{Snewss} \mbox{}\\

In this section, we generalize Theorem \ref{Tasss2} to odd characteristic using the techniques of \cite{VdGVdV}.  
For every odd prime $p$,
Corollary \ref{CKP} verifies that that there are infinitely many new values of $g$ for which there exists a supersingular curve of 
genus $g$ in characteristic $p$. 

\begin{corollary} \label{CKP} [Karemaker/Pries]
Let $p$ be prime.  
Let $\delta \in \NN$ be such that $0$ and $1$ are the only coefficients in the base $p$ expansion of $\delta$.
If $g=\delta p(p-1)^2/2$, then there exists a supersingular curve of genus $g$ defined over a finite field of characteristic $p$.
\end{corollary}

Remark: When $p=2$, then Corollary \ref{CKP} is the same as Theorem \ref{Tasss2} because the condition on $\delta$ is vacuous and $g=\delta$.
  
\begin{proof} 
The condition on $\delta$ implies that, for some $t \in \NN$,  
\begin{equation} \label{Gsum}
\delta = \sum_{i=1}^t p^{s_i} (1 + p + \cdots p^{r_i}), {\rm \ for \ some \ }  
r_i, s_i \in \ZZ^{\geq 0} {\rm \ such \ that \ } s_i \geq s_{i-1} +r_{i-1}+2.
\end{equation}

Let $u_i =(s_i+1)-\sum_{j=1}^{i-1} (r_j+1)$ and note $u_{i+1} \geq u_i + 1$.
Choose an $\FF_p$-linear subspace $L_i$ of dimension $d_i:=r_i+1$ in the vector subspace of 
$\overline{\FF}_p[x]$ of additive polynomials of degree $p^{u_i}$, with the requirement that
$L_i \cap L_j = \{0\}$ if $i \not = j$.
Let ${\mathbf L}= \oplus_{i=1}^t L_i$. 

If $f \in {\mathbf L}$, let $C_f:y^p-y=xf$.  Let $Y$ be the fiber product of $C_f \to {\mathbb P}^1$ 
for all $f \in {\mathbf L}-\{0\}$.
By \cite[Theorem B]{kanirosen}, ${\rm Jac}(Y)$ is isogenous to $\oplus_{f \not = 0} {\rm Jac}(C_f)$.
By Theorem \ref{Tasss}, ${\rm Jac}(C_f)$ is supersingular for each $f$ and so ${\rm Jac}(Y)$ is supersingular.

The genus of $Y$ is $g_Y = \sum_{f \not = 0} g_{C_f}$.
If $f \in {\mathbf L}$ is such that it has a non-zero contribution from $L_i$, but not from $L_j$ for $j >i$, then $g_{C_f} =p^{u_i}(p-1)/2$. 
There are $p^{d_i}-1$ non-zero polynomials in $L_i$.
The number of $f \in {\mathbf L}$ which have a non-zero contribution from $L_i$, but not from $L_j$ for $j >i$ is 
$(p^{d_i} -1) \prod_{j=1}^{i-1} p^{d_j}$.  
So
\begin{eqnarray*} 
g_Y & = & \sum_{i=1}^t (p^{d_i} -1) (\prod_{j=1}^{i-1} p^{d_j}) p^{u_i} (p-1)/2\\
& = & \sum_{i=1}^t (p-1)(p^{r_i} + \cdots + 1) p^{\sum_{j=1}^{i-1} (r_j+1)} p^{u_i-1} p(p-1)/2\\
& = & \sum_{i=1}^t (p^{r_i} + \cdots + 1) p^{s_i} p(p-1)^2/2 = \delta p(p-1)^2/2.
\end{eqnarray*}

%The proof when $p=2$.  Write
%\[g=2^{s_1}(1 + 2 + \cdots + 2^{r_1}) + 2^{s_2}(1 + 2 + \cdots 2^{r_2}) + \cdots + 2^{s_t}(1 + 2 + \cdots + 2^{r_t}).\]
%Insure that $s_i \leq s_{i-1} +r_{i-1}+2$.
%Let $u_i =(s_i+1)-\sum_{j=1}^{i-1} (r_j+1)$.  Note that $u_{i+1} \geq u_i + 1$.
%Choose a subspace $L_i$ of dimension $d_i:=r_i+1$ in the vector space of additive polynomials of degree $2^{u_i}$.
%Consider ${\mathbf L}= \oplus_{i=1}^t L_i$. 

%If $f \in {\mathbf L}$, let $C_f:y^p-y=xf$.  Let $Y$ be the fiber product of $C_f \to {\mathbb P}^1$ for all $f \in {\mathbf L}$.
%Then $J_Y$ is isogenous to $\oplus_{f \not = 0} J_{C_f}$, and thus supersingular.
%In particular, $g_Y = \sum_{f \not = 0} g_{C_f}$.
%There are $2^{d_i}-1$ polynomials which are non-zero in $L_i$.
%The number of $f \in {\mathbf L}$ which have a non-zero contribution from $L_i$, but not from $L_j$ for $j >i$ is 
%$(2^{d_i} -1) \prod_{j=1}^{i-1} 2^{d_j}$.  Each of these contributes $2^{u_i-1}$ to the genus.
%So 
%\[g_Y=\sum_{i=1}^t (2^{d_i} -1) \prod_{j=1}^{i-1} 2^{d_j} 2^{u_i-1}=\sum_{i=1}^t 2^{s_i}(1+ \cdots +2^{r_i})=g.\]
\end{proof}

\mbox{}\\ \section{Newton polygons and Ekedahl-Oort types of curves} \label{SNPEO}  \mbox{}\\

Suppose $X$ is a smooth projective curve of genus $g$ defined over an algebraically closed field $k$ of characteristic $p$.
Then $X$ can be classified by several invariants of its Jacobian, such as the $p$-rank, the Newton polygon, and the 
Ekedahl-Oort type.

The main goal of the section is to 
describe open questions about these invariants for Jacobians of smooth curves.
In this section, we use basic definitions and facts about these invariants, without defining them precisely,
to make the concepts and questions as clear as possible.
The technical definition of the invariants can be found in Section \ref{Snotation}.

\mbox{}\\ \subsection{Ordinary and supersingular elliptic curves} \label{Srevisit1} \mbox{}\\

To begin, we revisit the case of elliptic curves and describe the 
distinction between ordinary and supersingular elliptic curves from several other points of view.  

Let $E/k$ be an elliptic curve and let $\ell$ be prime.
The $\ell$-torsion group scheme $E[\ell]$ of $E$ is the kernel of the multiplication-by-$\ell$ morphism
$[\ell]:E \to E$.
Then 
\[
\#E[\ell](k) =
\begin{cases}
\ell^2 & \text{if } \ell \not = p\\
\ell & \text{if } \ell = p, E \ \text{ordinary} \\
1 & \text{if } \ell = p, E \ \text{supersingular}
\end{cases}.\]

The following conditions are equivalent to $E$ being ordinary: $E$ has $p$ points of order dividing $p$;
the Newton polygon of $E$ has slopes $0$ and $1$; or the group scheme $E[p]$ is isomorphic to $L:=\ZZ/p \oplus \mu_p$. 

The following conditions are equivalent to $E$ being supersingular:

\smallskip

{\bf (A)'} The only $p$-torsion point of $E$ is the identity: $E[p](k) = \{{\rm id}\}$.

\smallskip

{\bf (B)'} The Newton polygon of $E$ is a line segment of slope $1/2$.
\smallskip

{\bf (C)'} The group scheme $E[p]$ is isomorphic to $I_{1,1}$, the unique local-local symmetric ${\rm BT}_1$ group scheme of rank $p^2$.

\smallskip

Conditions (A)' and (B)' are equivalent by \cite[Theorem V.3.1 and page 142]{aec}.

More information about group schemes and condition (C)' can be found in \cite[Appendix A, Example 3.14]{G:book}.
Briefly, consider the group scheme $\alpha_p$ which is the kernel of Frobenius on $\GG_a$.
As a $k$-scheme, 
$\alpha_p \simeq {\rm Spec}(k[x]/x^p)$ with co-multiplication $m^*(x)=x \otimes 1 + 1 \otimes x$
and co-inverse ${\rm inv}^*(x)=-x$.
The group scheme $I_{1,1}$ fits in a non-split exact sequence
\begin{equation} \label{EI11}
0 \to \alpha_p \to I_{1,1} \to \alpha_p \to 0.
\end{equation}
Let $D_{1,1}$ be the mod $p$ Dieudonn\'e module of $I_{1,1}$, see Example \ref{exi11}.
%The image of $\alpha_p$ in \eqref{EI11} is the kernel of $F$ (Frobenius) and 
%(isomorphic to) the kernel of $V$ (Verschiebung).

\mbox{}\\ \subsection{Invariants for abelian varieties of higher dimension} \label{Sqinv} \mbox{}\\

Let $A$ be a principally polarized abelian variety of dimension $g$ over $k$.
For example, $A$ could be the Jacobian of a curve of genus $g$. 
Let $A[p]$ be the kernel of the multiplication-by-$p$ morphism of $A$.
The following conditions are all different for $g \geq 3$.

\smallskip

{\bf (A) $p$-rank $0$} -  The only $p$-torsion point of $A$ is the identity: $A[p](k) = \{{\rm id}\}$.

\smallskip

{\bf (B) supersingular} - The Newton polygon of $A$ is a line segment of slope $1/2$.

\smallskip

{\bf (C) superspecial} -  The group scheme $A[p]$ is isomorphic to $(I_{1,1})^g$.

\smallskip

\begin{proposition} \label{Pprankssss}
For conditions $(A),(B),(C)$ as defined above, there is an implication:
\[(C) \Rightarrow (B) \Rightarrow (A), \ 
{\rm but} \  (A) \stackrel{g \geq 3}{\not \Rightarrow} (B) \stackrel{g \geq 2}{\not \Rightarrow} (C).\] 
\end{proposition}

Proposition \ref{Pprankssss} is well-known to experts; 
for the proof, we need material about $p$-divisible groups from Section \ref{Snotation}.

\begin{proof}(Sketch)
\begin{enumerate}
\item{For the implication $(C) \Rightarrow (B)$:}
if the $p$-torsion of a $p$-divisible group $G$ satisfies (C), then 
$F^2 G \subset [p] G$.  By the basic slope estimate in \cite[1.4.3]{katzslope}, the slopes of the Newton polygon are all 
at least $1/2$; so the slopes all equal $1/2$, because the polarization forces the Newton polygon to be symmetric.
Thus $A$ is supersingular.
Alternatively, the implication $(C) \Rightarrow (B)$ follows from \cite[Theorem~2]{oort75} and \cite[Theorem~4.2]{O:sub}.
\item{For the non-implication $(B) \not \Rightarrow (C)$ when $g \geq 2$:}
an abelian variety can be isogenous 
but not isomorphic to a product of supersingular elliptic curves;
for example, quotients of a superspecial abelian variety by an $\alpha_p$-subgroup scheme have this property
when $g \geq 2$.
\item{For the implication $(B) \Rightarrow (A)$:} more generally,  
the $p$-rank of a $p$-divisible group 
is the multiplicity of the slope $0$ in the Newton polygon, so if all the slopes equal $1/2$, then the $p$-rank is $0$;
Alternatively, if $A$ is the Jacobian of a curve defined over a finite field, 
then the $p$-rank equals the number of roots of the 
$L$-polynomial that are $p$-adic units, which equals the multiplicity of the slope $0$ in the Newton polygon.
\item{For the non-implication $(A) \not \Rightarrow (B)$ when $g \geq 3$:}
there exists a principally polarized abelian variety 
whose Newton polygon has slopes $1/g$ and $(g-1)/g$; it has $p$-rank $0$ but is not supersingular when $g \geq 3$.
\end{enumerate}
\end{proof}

Here is the motivating question.

\begin{question} \label{Qmot1}
If $p$ is prime and $g \geq 2$, does there exist a smooth curve $X/\overline{\FF}_p$ of genus $g$
whose Jacobian (A) has $p$-rank $0$; (B) is supersingular; or (C) is superspecial?  
\end{question}

In Question \ref{Qmot1}, the answer to part (A) is yes for all $g$ and $p$, see Theorem \ref{TFVdG};
as seen in Section \ref{Ssupersingular}, the answer to 
part (B) is sometimes yes, but most often is not known;  the answer to part (C) 
is no unless $g$ is small relative to $p$, see Theorem \ref{Tekedahl}.

More generally, $A$ has the following invariants, defined more precisely
in Section \ref{Snotation}.

\smallskip

{\bf A. $p$-rank} - the integer $f$, with $0 \leq f \leq g$, such that $\#A[p](k)=p^{f}$.

\smallskip

{\bf B. Newton polygon} - the data of slopes for the $p$-divisible group $A[p^\infty]$.

\smallskip

{\bf C. Ekedahl-Oort type} -the isomorphism class of the symmetric ${\rm BT}_1$ group scheme $A[p]$.

\smallskip

The $p$-rank, Newton polygon, and Ekedahl-Oort type of a curve are defined to be those of its 
Jacobian.

Question \ref{Qmot1} asks whether the most rare $p$-rank, Newton polygon, or 
Ekedahl-Oort type occurs for the Jacobian of a smooth curve of genus $g$; here is a natural generalization 
of that question.

\begin{question} \label{Qmot2}
If $p$ is prime and $g \geq 2$, which $p$-ranks, Newton polygons, and Ekedahl-Oort types
occur for the Jacobians of smooth curves $X/\overline{\FF}_p$ of genus $g$?
\end{question}

 \mbox{}\\ \subsection{Small genus} \label{Smallg}\mbox{}\\

It is valuable to study Question \ref{Qmot2} for small $g$, because of the inductive techniques found in 
\cite{Pr:large} and Section \ref{Sinductive}.
When $g=2$ and $g=3$, the answer to Question \ref{Qmot2} is known for all $p$, 
because the open Torelli locus is open and dense
in the moduli space ${\mathcal A}_g$ of principally polarized abelian varieties of dimension $g$.

In this section, we include data for $g=2,3,4$, using some notation from Section \ref{Snotation}.
In particular, see Example \ref{Eir1} for the definition of $I_{r,1}$.
The tables in this section previously appeared in \cite{Pr:sg}.
Further examples for $4 \leq g \leq 11$ under congruence conditions on $p$ appear in \cite{LMPT}.  

 \mbox{}\\ \subsubsection{The case $g=2$} \mbox{}\\

The following table shows the 4 symmetric ${\rm BT}_1$ group schemes that occur
for principally polarized abelian surfaces.  They are listed by name, together with their
codimension in $\CA_2$, $p$-rank $f$, $a$-number $a$, Ekedahl-Oort type $\nu$, 
Young type $\mu$, Dieudonn\'e module, and Newton polygon slopes.  
Recall that $L = \ZZ/p \oplus \mu_p$.

\renewcommand{\arraystretch}{1.065}
\[\begin{array}{|l|c|c|c|c|l|r|r|}
\hline
\mbox{Name} & {\rm cod} & f & a &  \nu & \mu & \mbox{Dieudonn\'e module} & \mbox{Newton polygon} \\
% \omega & \mbox{cycle class (reduced)} \\ 
\hline
\hline
L^2 & 0 & 2 & 0 &  [1,2] & \emptyset & D(L)^2 & 0,0,1,1 \\
% s_2s_1s_2 & \lambda_0 \\
\hline
L \oplus I_{1,1} & 1 & 1 & 1 &  [1,1] & \{1\} & D(L) \oplus D_{1,1} & 0,\frac{1}{2},\frac{1}{2},1 \\
% s_1s_2 & (p-1)\lambda_1 \\
\hline
I_{2,1} & 2 & 0 & 1 &  [0,1] & \{2\} &{\mathbb E}/\EE(F^2+V^2) & \frac{1}{2},\frac{1}{2},\frac{1}{2},\frac{1}{2}\\
% s_2 & (p-1)(p^2-1)\lambda_2 \\ 
\hline
(I_{1,1})^2 & 3 & 0 & 2 &  [0,0] & \{2,1\} & (D_{1,1})^2 & \frac{1}{2},\frac{1}{2},\frac{1}{2},\frac{1}{2} \\
% 1 & (p-1)(p^2+1)\lambda_1\lambda_2 \\
\hline
\end{array}
\]

The last two rows contain all the supersingular objects.

The open Torelli locus $T(\CM_2)$ is open and dense in $\CA_2$.
From this, one can check that all 3 Newton polygons and all 4 Ekedahl-Oort types occur for Jacobians of smooth curves
of genus $2$ over $\overline{\FF}_p$ for all $p$, except for the following case:
there does not exist a superspecial smooth curve of genus $2$ over $\overline{\FF}_p$ when $p=2,3$.
This is a special case of \cite[Proposition 3.1]{IKO},
in which the authors determine the number of curves $X$ with ${\rm Jac}(X)[p] \simeq (I_{1,1})^2$. 

 \mbox{}\\ \subsubsection{The case $g=3$} \mbox{}\\

The following table shows the 8 symmetric ${\rm BT}_1$ group schemes that occur
for principally polarized abelian threefolds.  

\renewcommand{\arraystretch}{1.065}
\[\begin{array}{|l|c|c|c|c|l|r|}
\hline
\mbox{Name} & {\rm cod} & f & a & \nu & \mu &  \mbox{Dieudonn\'e module}  \\
% \omega \\
%\mbox{cycle class (reduced)}
\hline
\hline
L^3 & 0 & 3 & 0 &          [1,2,3] & \emptyset & D(L)^3 \\
% s_3s_2s_3s_1s_2s_3 \\
%\lambda_0
\hline
L^2 \oplus I_{1,1} & 1 & 2 & 1 & [1,2,2]  &  \{1\} & D(L)^2 \oplus D_{1,1} \\
% s_2s_3s_1s_2s_3 \\
%(p-1)\lambda_1\\
\hline
L \oplus I_{2,1} & 2 & 1 & 1&    [1,1,2] &  \{2\} & D(L) \oplus {\mathbb E}/\EE(F^2+V^2) \\
% s_3s_1s_2s_3 \\
% (p-1)(p^2-1)\lambda_2\\
\hline
L \oplus (I_{1,1})^2 & 3 & 1 & 2 & [1,1,1] &  \{2,1\} & D(L) \oplus (D_{1,1})^2 \\
% \{2,1\}
% s_1s_2s_3 \\ 
%(1-p)(p^2+1)\lambda_1\lambda_2 -2(p^3-1)\lambda_3 
\hline
I_{3,1} & 3 & 0 & 1 &  [0,1,2] & \{3\} & {\mathbb E}/\EE(F^3+V^3)  \\
% s_3s_2s_3 \\
%(p-1)(p^2-1)(p^3-1)\lambda_3 
\hline
I_{3,2} & 4 & 0 & 2 &          [0,1,1]  &  \{3,1\} & \EE/\EE(F^{2}+V) \oplus \EE/\EE(V^{2}+F)  \\
% s_2s_3 \\
% (p-1)^2(p^3+1)\lambda_1\lambda_3 
\hline
I_{1,1} \oplus I_{2,1} & 5 & 0 & 2 &   [0,0,1] &  \{3,2\} & D_{1,1} \oplus {\mathbb E}/\EE(F^2+V^2) \\
% s_3 \\
%(1-p)^2(p^6-1)\lambda_2\lambda_3 
\hline
(I_{1,1})^3 & 6 & 0 & 3 &      [0,0,0] &  \{3,2,1\}   & (D_{1,1})^3 \\
% 1 \\
% (p-1)(p^2+1)(p^3-1)\lambda_1\lambda_2\lambda_3\\
\hline
\end{array}
\]

The objects in the last two rows are always supersingular but the situation for $I_{3,1}$ and $I_{3,2}$ is more
subtle.
By \cite[Theorem 5.12]{O:hypsup},
if $A[p] \simeq I_{3,1}$, then the $p$-divisible group is usually isogenous to $G_{1,2} \oplus G_{2,1}$ 
(slopes $1/3, 2/3$) 
but it can also be isogenous to $G_{1,1}^3$ (supersingular).
This shows that the Ekedahl-Oort stratification does not refine the Newton polygon stratification
for $g \geq 3$.
  
The open Torelli locus $T(\CM_3)$ is open and dense in $\CA_3$.
From this, one can check that all 5 Newton polygons and all 8 Ekedahl-Oort types occur for 
Jacobians of smooth curves over $\overline{\FF}_p$, except when $p=2$ for $(I_{1,1})^3$ and possibly when $p=2$ for $I_{1,1} \oplus I_{2,1}$.

Here are some references for the 4 bottom rows of the table, which are the $p$-rank $0$ cases.
There exists a smooth curve $X$ of genus $3$ over ${\overline \FF}_p$ such that ${\rm Jac}(X)$ has the given $p$-torsion group scheme:
\begin{enumerate}
\item $I_{3,1}$, for all $p$ by \cite[Theorem 5.12(2)]{O:hypsup};
\item $I_{3,2}$, \cite[Lemma 4.8]{Pr:large} for $p \geq 3$ and \cite[Example 5.7(3)]{EP} for $p=2$;
\item $I_{1,1} \oplus I_{2,1}$, \cite[Lemma 4.8]{Pr:large} for $p \geq 3$ 
(using \cite[Proposition 7.3]{O:strat});\\ 
when $p =2$, this group scheme does not occur as the $2$-torsion of a hyperelliptic curve
but no one has checked whether it occurs for a smooth plane quartic;
\item $(I_{1,1})^3$, if and only if $p \geq 3$ by \cite[Theorem 5.12(1)]{O:hypsup}.
\end{enumerate}

 \mbox{}\\ \subsubsection{The case $g=4$} \mbox{}\\

There are some open questions even for curves of genus $4$.
A table for the $16$ symmetric $\bt_1$ group schemes of rank $p^8$ appears in \cite[Section 4.4]{Pr:sg}.  
  
\begin{question} \label{Qss4}
For all $p$, does there exist a smooth curve of genus $4$ which is supersingular?
\end{question}

We give an affirmative answer to Question \ref{Qss4}
when $p \equiv 2 \bmod 3$ or $p \equiv 2, 3, 4 \bmod 5$ in \cite{LMPT}.

\begin{question} \label{Q1234}
For all $p$, does there exist a smooth curve whose 
Newton polygon has slopes $1/3, 1/2, 2/3$
(whose $p$-divisible group is isogenous to 
$G_{1,2} \oplus G_{2,1} \oplus G_{1,1}$)?
\end{question}
 
We give an affirmative answer to Question \ref{Q1234}
when $p \equiv 4,7 \bmod 9$ in \cite{LMPT}.

\begin{question} \label{Qa24}
For all $p$, does there exist a smooth curve of genus $4$ with $p$-rank $0$ and $a$-number at least $2$?
\end{question}

More generally, it is not known whether the following Young types occur for Jacobians of smooth curves of genus $4$ 
for all $p$:
\[\{4,1 \}, \ \{4,2 \}, \ \{4,3 \}, \{4,2,1 \}, \ \{4,3,1 \}, \{4,3,2 \}, \ \{4,3,2,1 \}.\]

\mbox{}\\ \subsection{Non-existence results for Jacobians} \label{Snonext}   \mbox{}\\

\mbox{}\\ \subsubsection{Non-existence results for Ekedahl-Oort types of Jacobians}   \mbox{}\\

In some cases, Question \ref{Qmot1} has a negative answer.
This is the only non-existence result currently known for Question \ref{Qmot2}.
Recall that $X$ is superspecial if ${\rm Jac}(X)[p]$ is isomorphic to $(I_{1,1})^g$.

\begin{theorem} \label{Tekedahl} \cite{ekedahl87}, see also \cite{Baker}
If $X/\overline{\FF}_p$ is a superspecial curve of genus $g$, then $g \leq p(p-1)/2$.
\end{theorem}

Theorem \ref{Tekedahl} can be stated as a non-existence result: 
a smooth curve of genus $g$ defined over 
$\overline{\FF}_p$ cannot be superspecial if $g > p(p-1)/2$.
The Hermitian curve $H_p:y^p+y=x^{p+1}$ is superspecial and its genus 
realizes the bound in Theorem \ref{Tekedahl}.
In \cite{Re} and \cite{zhou}, 
there are generalizations of Theorem \ref{Tekedahl} which bound the difference between 
$g$ and the $a$-number for Jacobians of smooth curves when $g$ is large relative to $p$.

\mbox{}\\ \subsubsection{The $p$-rank of curves with prime-to-$p$ automorphisms} \mbox{}\\

The automorphism group of a curve can place restrictions on its $p$-rank.
If $H \subset {\rm Aut}(X)$, consider the cover
$\phi_H: X \to Z$, where $Z=X/H$ is the quotient curve.  
First, we consider the case that $p \nmid |H|$.

Bouw proved that there are non-trivial constraints on the $p$-rank of $X$, arising from 
the action of the Frobenius operator $F$ on the $H$-module $H^1(X, {\mathcal O})$. 
We describe this result in the case that $Z$ is the projective line and $H \simeq \ZZ/d$.
In this case, there is a formula for the dimensions of the $\ZZ/d$-eigenspaces 
$L_i$ of $H^1(X, {\mathcal O})$, for $1 \leq i \leq d-1$, in terms of the inertia type of $\phi_H$.
The action of $F$ permutes the eigenspaces $L_i$; the congruence class of $p$ modulo $d$ determines
which eigenspaces are in the same orbit under $F$.
Bouw's key observation is that the stable rank of $F$ on $L_i$ is bounded by the minimum of the
ranks of $F$ among all the eigenspaces in the orbit of $L_i$.
 
This leads to the following result.

\begin{theorem} (See \cite[page 300, Theorem 6.1]{Bouw} for precise version)
Suppose $\phi:X \to {\mathbb P}^1$ is a cyclic degree $d$ cover with $p \nmid d$.
Then the inertia type of $\phi$ determines a (typically non-trivial) upper bound $B$ for the $p$-rank of $X$.
Under mild conditions, this upper bound $B$ occurs as the $p$-rank of a cover $\phi$ of this inertia type.
\end{theorem}
  
 \mbox{}\\ \subsubsection{The $p$-rank of curves in the context of $p$-group covers} \mbox{}\\

There are significant restrictions on the $p$-rank of curves that have an automorphism of order $p$.
If $x \in X$, let $I_x$ be the inertia group of $\phi_H$ at $x$ and let $e_x=|I_x|$.

\begin{theorem} \label{TDS} (Deuring-Shafarevich formula) \cite[Theorem 4.2]{subrao}, see also \cite[Corollary~1.8]{Crew}.
Suppose that $\phi_H: X \to Z$ is an $H$-Galois cover where $H$ is a $p$-group 
and let $f_X, f_Z$ be the $p$-ranks of $X$ and $Z$ respectively.  Then 
$f_X - 1 = |H|(f_Z-1) + \sum_{x \in X} (e_x -1)$.
\end{theorem}

Here is a consequence of the Deuring-Shafarevich formula.
Suppose $H$ is a $p$-group and $\phi_H$ is a cover of the projective line, 
branched just at one point, say $\infty$.
Since the \'etale fundamental group of the affine line is trivial, 
there is a unique point in $\phi_H^{-1}(\infty)$, and $\phi_H$ is totally ramified at that point.  
Then Theorem \ref{TDS} implies that $X$ has $p$-rank $0$.

 \mbox{}\\ \subsubsection{The Newton polygon of curves in the context of wild ramification} \mbox{}\\

The cover $\phi_H:X \to Z$ is wildly ramified 
if $p$ divides the order of the inertia group at any of the ramification points of $\phi_H$.
Wild ramification places restrictions on the Newton polygon.
For example, when $p=2$, there does not exist a hyperelliptic supersingular curve of genus $3$ by \cite[5.6.2]{O:hypsup}.
This can be generalized as follows.

\begin{theorem}
\begin{enumerate}
\item \cite[Theorem 1.2]{scholtenzhu} 
If $p=2$ and $n \geq 2$, then there does not exist a supersingular curve of genus $2^n-1$ 
which is hyperelliptic.

\item \cite[Theorem 2]{blacheprevaluations} If $p$ is prime and $n \in \NN$ with $n(p-1) > 2$ and $1 \leq i \leq p-1$, 
then there does not exist a supersingular curve 
of genus $(ip^n - i -1)(p-1)/2$ which is a $p$-cyclic cover of the projective line.
\end{enumerate}
\end{theorem}

For more information about the connection between the Newton polygons of Artin-Schreier curves and exponential sums, 
see \cite{blache1, blache2, blacheprefirstvertices, farnellpries, 
scholtenzhu02C, scholtenzhu02B, scholtenzhu03}.

\mbox{}\\ \subsubsection{The Ekedahl-Oort type of curves in the context of wild ramification}   \mbox{}\\

Wild ramification also places restrictions on the Ekedahl-Oort type.
When $p=2$, there is a complete classification in \cite{EP} of the Ekedahl-Oort types which occur for Jacobians of hyperelliptic curves.  
Briefly, each hyperelliptic curve in characteristic $2$ has an affine equation of the form $y^2+y=h(x)$ for a rational function $h(x)$.
Write ${\rm div}_\infty(h(x))=\sum_{i=0}^{f} c_i P_i$.  
Then the multi-set $\{c_i\}$ determines the Ekedahl-Oort type \cite[Theorems~1.2-1.3]{EP}.
Thus, when $p=2$, the number of Ekedahl-Oort types that occur for Jacobians of hyperelliptic curves of genus $g$ 
is approximately the number of partitions of $g+1$.

%*********************************************************************************************************************

 \mbox{}\\ \section{Notation and background} \label{Snotation} \mbox{}\\

This section contains a very concise summary of the notation and background used in this paper.
The reader who would like to learn this topic thoroughly should consult a more complete description, 
for example as found in \cite{LO}, \cite{O:strat}, or the chapter 
{\it Moduli of Abelian Varieties} by Chai and Oort in this volume.

Let $k$ be an algebraically closed field of characteristic $p>0$.
Let $A$ be a principally polarized abelian variety of dimension $g$ defined over $k$.

 \mbox{}\\ \subsection{The $p$-torsion group scheme} \mbox{}\\

The multiplication-by-$p$ morphism $[p]:A \to A$ is a finite flat morphism of degree $p^{2g}$.
It is the composition of $\frob$ and $\ver$, where $\frob:A \to A^{(p)}$ denotes the relative Frobenius morphism
and $\ver: A^{(p)} \to A$ is the Verschiebung morphism.
The morphism $\frob$ comes from the $p$-power map on the structure sheaf; it is purely inseparable of degree $p^g$.  
Also $V$ is the dual of $\frob_{A^{\rm dual}}$.  

The {\it $p$-torsion group scheme} of $A$ is 
\[A[p]= \ker[p].
\]
In fact, $A[p]$ is
a symmetric $\bt_1$ group scheme as defined in \cite[2.1, Definition 9.2]{O:strat}.  It has rank $p^{2g}$.  It is killed by $[p]$, with
$\ker(\frob) = \im(\ver)$ and $\ker(\ver) = \im(\frob)$.

The principal polarization on $A$ induces a principal quasipolarization (pqp) on $A[p]$,
i.e., an anti-symmetric isomorphism $\psi:A[p] \to A[p]^D$, where $D$ denotes the Cartier dual.  
(This definition must be modified slightly if $p=2$.)
Thus, $A[p]$ is a symmetric $\bt_1$ group scheme together with a principal quasipolarization.

Isomorphism classes of pqp $\bt_1$ group schemes over $k$
have been completely classified in terms of Ekedahl-Oort types \cite[Theorem 9.4
\& 12.3]{O:strat}, see Section \ref{Seotype}.  This builds on
work of Kraft \cite{Kraft} (unpublished, which did not include polarizations) and of Moonen
\cite{M:group} (for $p \geq 3$). 
(When $p=2$, there are complications with the polarization which are resolved in \cite[9.2, 9.5, 12.2]{O:strat}.)

 \mbox{}\\ \subsection{The $p$-rank and $a$-number} \label{Sprankanumber} \label{Sanumber} \mbox{}\\

The {\it $p$-rank} of $A$ is 
\[f={\rm dim}_{\FF_p} {\rm Hom}(\mu_p, A),\]
where $\mu_p$ is the kernel of Frobenius on $\GG_m$.
The advantage of this definition is that it is also valid for semi-abelian varieties.
When $A$ is an abelian variety, then $p^f$ is the cardinality of $A[p](k)$; 
the reason is that the multiplicity of the group schemes $\ZZ/p$ and $\mu_p$ in $A[p]$ is the same 
because of the symmetry induced by the polarization. 

The {\it $a$-number} of $A$ is 
\[a={\rm dim}_k {\rm Hom}(\alpha_p, A),\]
where $\alpha_p$ is the kernel of Frobenius on $\GG_a$.
It is known that $0 \leq f \leq g$ and $1 \leq a +f \leq g$.

Since $\mu_p$ and $\alpha_p$ are both simple group schemes,
the $p$-rank and $a$-number are additive;
\begin{equation}
\label{eqfadditive}
f(A_1\oplus A_2) = f(A_1)+f(A_2)\text{ and }a(A_1\oplus A_2) = a(A_1)+a(A_2).
\end{equation}

The $p$-rank and $a$-number can also be defined for a $p$-torsion group scheme, 
$p$-divisible group, or Dieudonn\'e module.

 \mbox{}\\ \subsection{The $p$-divisible group and Newton polgon} \mbox{}\\

For each $n \in \NN$, consider the multiplication-by-$p^n$ morphism $[p^n]:A \to A$ and its kernel $A[p^n]$.
The $p$-divisible group of $A$ is $A[p^\infty] = \varinjlim A[p^n]$.

For each pair $(c,d)$ of non-negative relatively prime integers, fix
a $p$-divisible group $G_{c,d}$ of codimension $c$, dimension $d$, and thus height $c+d$.
By the Dieudonn\'e-Manin classification \cite{maninthesis}, there is an isogeny of $p$-divisible groups 
\[A[p^\infty] \sim \oplus_{\lambda=\frac{d}{c+d}} G_{c,d}^{m_\lambda},\] where $(c,d)$ ranges over pairs of 
non-negative relatively prime integers.

The Newton polygon is the multi-set of values of $\lambda$.
The line segments with slopes $\lambda$ form a lower convex hull from $(0,0)$ to $(2g,g)$.
The Newton polygon is an isogeny invariant of $A$; it is determined by the multiplicities $m_\lambda$.

The abelian variety $A$ is {\it supersingular} if and only if $\lambda=1/2$ is the only slope of 
its $p$-divisible group $A[p^\infty]$.
Let $\sigma_g$ denote the supersingular Newton polygon of height $2g$.
If $G_{1,1}$ denotes the $p$-divisible group of dimension $1$ and height $2$,
then $A$ is supersingular if and only $A[p^\infty] \sim G_{1,1}^g$.

There is a partial ordering on Newton polygons: 
one Newton polygon is smaller than a second if the lower convex hull of the first is never below the second.
The Newton polygon $\sigma_g$ is the smallest in this partial ordering.

 \mbox{}\\ \subsection{Dieudonn\'e modules}
\label{subsecdefcartier} \mbox{}\\

The $p$-divisible group $A[p^\infty]$ and the $p$-torsion group scheme $A[p]$ can be described using covariant Dieudonn\'e theory, see e.g., \cite[15.3]{O:strat}.  
Differences between the covariant and contravariant theory 
do not cause a problem in this paper since all objects we consider are 
principally quasipolarized and thus symmetric.

Briefly, 
let $\sigma$ denote the Frobenius automorphism of $k$ and its lift to the Witt vectors $W(k)$;
let $\til \EE = \til\EE(k) = W(k)[F,V]$ denote the non-commutative ring generated by semilinear operators $F$ and $V$ with relations
\begin{equation} \label{Efv}
FV=VF=p, \ F \tau = \tau^\sigma F, \ \tau V=V \tau^\sigma,
\end{equation} for all $\tau \in W(k)$.  

There is an equivalence of categories $\dieu_*$
between $p$-divisible groups over $k$ and $\til\EE$-modules which are free of finite rank over $W(k)$.
For example, the Dieudonn\'e module $D_\lambda :=  \dieu_*(G_{c,d})$ is a free $W(k)$-module of rank $c+d$.
Over $\operatorname{Frac}W(k)$, there is a basis $x_1, \ldots, x_{c+d}$ for $D_\lambda$ such that $F^{d}x_i=p^c x_i$.

Similarly, let $\EE = \til \EE \otimes_{W(k)} k$ be the reduction of the Cartier ring modulo $p$; it is a non-commutative ring $k[F,V]$ subject to the same constraints as \eqref{Efv}, except that $FV = VF = 0$ in $\EE$.  Again, there is an equivalence of categories $\dieu_*$ between finite commutative group schemes $I$ (of rank $2g$) annihilated by $p$ and $\EE$-modules of finite dimension ($2g$) over $k$.

If $M = \dieu_*(I)$ is the Dieudonn\'e module over $k$ of $I$, then  
a principal quasipolarization  $\psi:I \to I^D$ induces a 
a nondegenerate symplectic form $\ang{\cdot,\cdot}:M \times M \to k$
on the underlying $k$-vector space of $M$, subject to the additional constraint that, for all $x$ and $y$ in $M$,
\[\ang{Fx,y} = \ang{x,Vy}^\sigma.\]

If $A$ is the Jacobian of a curve $X$, then there is an isomorphism of $\EE$-modules between the {\it contravariant} 
Dieudonn\'e module over $k$ of ${\rm Jac}(X)[p]$ 
and the de Rham cohomology group $H^1_{\rm dR}(X)$ by \cite[Section 5]{Oda}.  
The canonical principal polarization on $\operatorname{Jac}(X)$ induces a canonical isomorphism $\dieu_*(\operatorname{Jac}(X)[p]) \simeq H^1_{\rm dR}(X)$; we will use this identification without further comment.

For elements $A_1, \ldots, A_r \in \EE$, 
let $\EE(A_1, \ldots, A_r)$ denote the left ideal $\sum_{i=1}^r \EE A_i$ of $\EE$ generated by $\{A_i \mid 1 \leq i \leq r\}$.

\begin{example}\label{exi11} {\em The group scheme $I_{1,1}$.}
There is a unique symmetric ${\rm BT}_1$ group scheme
of rank $p^2$ and $p$-rank $0$, which we denote $I_{1,1}$.
It is a non-split extension of $\alpha_p$ by $\alpha_p$ as in \eqref{EI11}.
The mod $p$ Dieudonn\'e module of $I_{1,1}$ is $D_{1,1} := \dieu_*(I_{1,1})$.  
Then $D_{1,1} \simeq \EE/\EE(F+V)$.
\end{example}

%%%%%%%%%%%%%%%%%%%%%%%%%%%%%%%%%%%%%%%%%%%%
 \mbox{}\\ \subsection{The Ekedahl-Oort type} \label{Seotype} \mbox{}\\
%%%%%%%%%%%%%%%%%%%%%%%%%%%%%%%%%%%%%%%%%%%%

As in \cite[Sections 5 \& 9]{O:strat}, the isomorphism type of a symmetric ${\rm BT}_1$ group scheme 
$I$ over $k$ can be encapsulated into combinatorial data.
If $I$ is symmetric with rank $p^{2g}$, then there is a {\it final filtration} $N_1 \subset N_2 \subset \cdots \subset N_{2g}$ 
of ${\mathbb D}_*(I)$ as a $k$-vector space which is stable under the action of $V$ and $F^{-1}$ such that $i={\rm dim}(N_i)$ \cite[5.4]{O:strat}.
%If $w$ is a word in $V$ and $F^{-1}$, then $wD(G)$ is an object in the filtration; in particular, $N_g = V D(G) =F^{-1}(0)$.

The {\it Ekedahl-Oort type} of $I$ is 
\[\nu=[\nu_1, \ldots, \nu_g], \ {\rm where} \ {\nu_i}={\rm dim}(V(N_i)).\]
The $p$-rank is ${\rm max}\{i \mid \nu_i=i\}$ and the $a$-number equals $g-\nu_g$.
%The Ekedahl-Oort type of $G$ does not depend on the choice of a final filtration.
There is a restriction $\nu_i \leq \nu_{i+1} \leq \nu_i +1$ on the Ekedahl-Oort type.
There are $2^g$ Ekedahl-Oort types of length $g$ since all sequences satisfying this restriction occur.   
By \cite[9.4, 12.3]{O:strat}, there are bijections between (i) Ekedahl-Oort types of length $g$; (ii) 
pqp ${\rm BT}_1$ group schemes over $k$ of rank $p^{2g}$;
and (iii) pqp Dieudonn\'e modules of dimension $2g$ over $k$.

By \cite{EVdG}, the Ekedahl-Oort type can also be described by its Young type $\mu$.  
Given $\nu$, for $1 \leq  j \leq g$, consider the strictly decreasing sequence
\[\mu_j = \#\{i \mid 1 \leq i \leq  g, \  i - \nu_i \geq j\}.\] 
There is a Young diagram with $\mu_j$ squares in the $j$th row. The {\it Young type}
is $\mu = \{\mu_1, \mu_2, . . .\}$, where one eliminates all $\mu_j$ which are $0$. 
The $p$-rank is $g - \mu_1$ and the $a$-number is $a = {\rm max}\{j \mid \mu_j \not = 0\}$.

The Ekedahl-Oort type places restrictions on the Newton polygon and vice-versa, see \cite{harashita07, harashita10}.
  
\begin{example} \label{Eir1}
Let $r \in \NN$.  
There is a unique symmetric $\bt_1$ group scheme of rank $p^{2r}$
with $p$-rank $0$ and $a$-number $1$, which we denote $I_{r,1}$. 
The Dieudonn\'e module of $I_{r,1}$ has the property that $\dieu_*(I_{r,1}) \simeq \EE/\EE(F^r + V^r)$.
For $I_{r,1}$, the Ekedahl-Oort type is $[0,1,2,\ldots, r-1]$ and the Young type is $\{r\}$.
\end{example}

\mbox{}\\ \section{Stratifications of moduli spaces} \label{Sstratmoduli} \mbox{}\\

Let $\CA_g=\CA_g \otimes \FF_p$ be the 
moduli space of principally polarized abelian varieties of dimension $g$ in characteristic $p$
and let $\tilde{\CA}_g$ denote its toroidal compactification.
Let ${\mathcal M}_g= {\mathcal M}_g \otimes \FF_p$ 
denote the moduli space of smooth curves of genus $g$ in characteristic $p$
and let $\overline{\mathcal M}_g$ denote the Deligne-Mumford compactification.

There are stratifications of $\CA_g$ by $p$-rank, Newton polygon, and Ekedahl-Oort type.
We do not discuss the major results about these stratifications on $\CA_g$ in this paper, but merely
focus on results about the dimension of the strata and provide important references in Section \ref{Sdimstrata}.
In Section \ref{Sisogeny}, we include some information about central and isogeny leaves which is not used
later in the paper.

Consider the Torelli morphism $T: {\overline \CM}_g \to \tilde{\CA_g}$.
The open Torelli locus is the image of $\CM_g$ under $T$.
Via the Torelli morphism, the moduli space $\CM_g$ also has stratifications by 
these invariants. This leads to a geometric generalization of Question \ref{Qmot2}.

\begin{question} \label{Qmot3}
If $p$ is prime and $g \geq 2$, does the open Torelli locus intersect the strata of $\CA_g$ by 
$p$-rank, Newton polygon, or Ekedahl-Oort type? If so, what are the geometric properties of the intersection?
\end{question}

Unfortunately, some of the powerful techniques used to study the stratifications on 
$\CA_g$ are not available on $\CM_g$, 
such as deformation tools 
(Serre-Tate theory and Dieudonn\'e theory) and Hecke operators.
The Torelli morphism is a good tool, but it is not
flat since the fibers have positive dimension over the boundary $\partial \CM_g$, whose points
represent singular curves of genus $g$.  

In Section \ref{Sunlikely}, we discuss expectations about Question \ref{Qmot3} from the perspective of unlikely intersections. 
Finally, in Sections \ref{SFVdG}-\ref{SFVdGproof}, we sketch a proof of Faber and Van der Geer's theorem
about the dimensions of the $p$-rank strata on $\CM_g$.

\mbox{}\\ \subsection{Dimensions of the strata} \label{Sdimstrata} \mbox{}\\

Let $g \geq 1$.  The dimension of $\CA_g$ is $g(g + 1)/2$.  Here is some information 
about the dimensions of the strata plus a partial list of some valuable references. 

\smallskip

{\bf (A) The $p$-rank strata:}

For $0 \leq f \leq g$, let $\CA_g^f$ denote the $p$-rank $f$ stratum
whose points represent curves of genus $g$ and $p$-rank $f$.
By \cite{normanoort}, $\CA_g^f$ is non-empty and pure of codimension $g-f$.

\smallskip

Some valuable references: 

Oort, {\it Subvarieties of moduli spaces} \cite{O:sub}

Norman and Oort, {\it Moduli of abelian varieties} \cite{normanoort}

\smallskip

{\bf (B) Newton polygon strata:}

Let $\xi$ be a symmetric Newton polygon of height $2g$.
Consider the stratum $\CA_g[\xi]$ of $\CA_{g}$ 
whose points represent principally polarized abelian varieties with Newton polygon $\xi$.
As in \cite[3.3]{OortNPformalgroups} or \cite[1.9]{OortNPstrata}, define 
\[{\rm sdim}(\xi)=\# \Delta(\xi), \]
where
\[\Delta(\xi) = \{(x,y) \in \ZZ \times \ZZ \mid y < x \leq g, \ (x,y) {\rm \ on \ or \ above\ } \xi\}.\]

By \cite[Theorem 4.1]{OortNPstrata}, the dimension of $\CA_g[\xi]$ is 
\[{\rm dim}(\CA_g[\xi]) = {\rm sdim}(\xi).\]

By \cite{chaioort}, $\CA_g[\xi]$ is irreducible if $\xi$ is not the supersingular Newton polygon $\sigma_g$.
This implies that $\CA_g^f$ is irreducible, except when $g=1,2$ and $f=0$.

\smallskip

Some valuable references:

Katz, {\it Slope filtration of $F$-crystals}, \cite{katzslope}

de Jong and Oort, {\it Purity of stratification by Newton polygons} \cite{oortpurity}

Chai and Oort, {\it Monodromy and irreducibility of leaves} \cite{chaioort}

\smallskip

{\bf (C) Ekedahl-Oort strata:}

Let $\xi$ be a symmetric $\bt_1$ group scheme with Ekedahl-Oort type $\nu=[\nu_1, \ldots, \nu_g]$.
By \cite[Theorem 1.2]{O:strat}, the stratum of $\CA_g$ whose points represent abelian varieties with Ekedahl-Oort type $\nu$
is locally closed and quasi-affine with dimension $\sum_{i=1}^g \nu_i$.

\smallskip

Some valuable references:

Kraft, {\it Kommutative algebraische $p$-Gruppen} \cite{Kraft}

Oort, {\it A stratification of a moduli space of abelian varieties} \cite{O:strat}

Moonen and Wedhorn, {\it Discrete invariants of varieties in positive characteristic} \cite{moonenwedhorn}

Ekedahl and Van der Geer, {\it Cycle classes of the E-O stratification on the moduli of abelian varieties} \cite{EVdG}

\smallskip

\mbox{}\\ \subsection{Some information about central leaves and isogeny leaves} \label{Sisogeny} \mbox{}\\

As before, let $\xi$ be a symmetric Newton polygon of height $2g$. 
The geometric structure of the Newton polygon stratum $\CA_g[\xi]$ can be studied using central and isogeny leaves.

At a point $x$ representing a principally polarized abelian variety $A$ with Newton polygon $\xi$, 
one can define the central leaf $C(x)$ and isogeny leaf $I(x)$ of $\CA_g[\xi]$ at $x$.
By \cite[Theorem 5.3]{oortjams}, $\CA_g[\xi]$ has an `almost' product structure
by $C(x)$ and $I(x)$.
Informally, $C(x)$ is the subset of $\CA_g[\xi]$ whose points represent abelian varieties $A'$ such that 
there is an isomorphism of $p$-divisible groups $A'[p^\infty] \simeq A[p^\infty]$. 
To define $I(x)$, first define $H_\alpha(x)$ to be the set of points of $\CA_g$ connected to $x$ by iterated 
$\alpha_p$-isogenies (over extension fields).
Then $I(x)$ is the union of all irreducible components of $H_\alpha(x)$ that contain $x$.

The dimensions of the
central leaf and the isogeny leaf at $x$ depend only on the slopes of $\xi$.
Let $c(x) = {\rm dim}(C(x))$ and $i(x) = {\rm dim}(I(x))$.
By \cite[Corollary 5.8]{oortjams}, the `almost' product structure implies
that
\[c(x) + i(x)={\rm sdim}(\xi).\]

The formula for $c(x)$ is more complicated than that of ${\rm sdim}(\xi)$, but depends only on $\xi$.  It relies on counting lattice points $(x,y)$
in a region determined by the slopes $\lambda$ and $1-\lambda$ of $\xi$. 
We include the formula in a special case.
Suppose $\xi$ has slopes $m/(m+n)$ and $n/(m+n)$ where ${\rm gcd}(m,n)=1$ and $m > n$.
Then \[c(\xi) = \#\{(x,y) \in \ZZ \times \ZZ \mid 0 \leq x \leq g, \frac{n}{m+n}x \leq y < \frac{m}{m+n} x\}.\]

\begin{example} \label{Ecidim}
when $g=4$ and $\xi$ has slopes $1/4, 3/4$, then 
${\rm sdim}(\xi)= 6$, $c(\xi)=5$ and $i(\xi)=1$;
when $g=5$ and $\xi$ has slopes $2/5, 3/5$, then 
${\rm sdim}(\xi) = 7$, $c(\xi) = 3$, and $i(\xi)=4$.
%\cite[Theorem 3.13]{oortjams}.
\end{example}

\mbox{}\\ \subsection{Unlikely intersections} \label{Sunlikely} \mbox{}\\

In \cite[Expectation 8.5.4]{oortpadova05}, Oort made the following observation.
The moduli space $\CA_g$ has dimension $g(g + 1)/2$
and its supersingular locus $\CA_g[\sigma_g]$ has dimension $\lfloor g^2/4 \rfloor$.
The difference $\delta_g:=g(g + 1)/2 - \lfloor g^2/4 \rfloor$ is the length of a chain 
which connects the ordinary Newton polygon $\nu_g$ to the 
supersingular Newton polygon $\sigma_g$ in the partial ordered set of Newton polygons.

\begin{remark} \label{Rlongchain}
If $g \geq 9$, then $\delta_g > 3g - 3 ={\rm dim}(\CM_g)$.
\end{remark}

Because of Remark \ref{Rlongchain}, at least one of the following is true:
\begin{enumerate}
\item Either $\CM_g$ does not admit a perfect stratification by Newton polyon: this means that 
there are two Newton polygons $\xi_1$ and $\xi_2$ such that $\CA_g[\xi_1]$ is in the closure of $\CA_g[\xi_2]$,
but $\CM_g[\xi_1]$ is not in the closure of $\CM_g[\xi_2]$;
\item or some Newton polygons do not occur for Jacobians of smooth curves.
\end{enumerate}

At this time, no Newton polygon has been excluded from occurring for a Jacobian in any characteristic.

One expects, see \cite[Conjecture 8.5.7]{oortpadova05}, that decomposable $p$-divisible groups 
will be more tractable to study.  For this reason, one expects that Newton polygons whose slopes have small 
denominators will be easier to realize for Jacobians of smooth curves.

Among the indecomposable ones, one test case
is when $g=11$ with $p$-divisible group isogenous to $G_{5,6} \oplus G_{6,5}$ (slopes $5/11, 6/11)$
\cite[Expectation 8.5.3]{oortpadova05}.
This Newton polygon stratum has high codimension in ${\mathcal A}_{11}$, so one expects
that it would not intersect the open Torelli locus.  
However, when $p=2$, Blache remarked that the Newton polygon of the curve 
$y^2+y = x^{23}+x^{21}+x^{17}+x^7+x^5$ 
has slopes $5/11, \ 6/11$; another example of rare slopes occurs 
for the curve $y^2+y=x^{25}+x^9$ in characteristic $2$ whose Newton polygon has slopes $5/12, \ 7/12$.

\mbox{}\\ \subsection{Results about the $p$-rank stratification} \label{SFVdG} \mbox{}\\

In this section, we explain how the $p$-rank strata do have the expected dimension in the moduli space $\CM_g$
of curves of genus $g$.
Fix a prime $p$ and integers $g \geq 2$ and $f$ such that $0 \leq f \leq g$. 

The moduli space $\CM_g$ can be stratified by $p$-rank into strata $\CM_g^f$
whose points represent curves of genus $g$ and $p$-rank $f$.
Similarly, one can stratify the moduli space $\CH_g$ of hyperelliptic curves
or the compactifications $\overline{\CM}_g$ and $\overline{\CH}_g$ by $p$-rank.

Recall that $\CA_g^f$ is irreducible unless $g=1,2$ and $f=0$.
In most cases, it is not known whether $\CM_g^f$ and $\CH_g^f$ are irreducible. 

\begin{theorem} \label{TFVdG} \cite[Theorem 2.3]{fabervandergeer}
Let $g \geq 2$.  Every component of $\overline{\CM}_g^f$ has dimension $2g-3+f$ 
(codimension $g-f$ in $\overline{\CM}_g$);
in particular, there exists a smooth curve over $\overline{\FF}_p$ with genus $g$ and $p$-rank $f$. 
\end{theorem}

\begin{theorem} ($p$ odd) \cite[Theorem 1]{glasspries}, see also \cite[Lemma 3.1]{AP:hyp}, 
($p=2$) \cite[Corollary 1.3]{prieszhu}
Every component of $\overline{\CH}_g^f$ has dimension $g-1+f$ 
(codimension $g-f$ in $\overline{\CH}_g$);
in particular, there exists a smooth hyperelliptic curve over $\overline{\FF}_p$ 
with genus $g$ and $p$-rank $f$.
\end{theorem}

In \cite{AP:mono} and \cite{AP:hyp}, 
the authors prove more about the components of $\overline{\CM}_g^f$ and $\overline{\CH}_g^f$;
this includes results about how the components intersect the boundary and  
results about the $\ell$-adic monodromy of the components.
In \cite{Pr:large}, for all $g \geq 3$ and all $p$, there are results
about the moduli of curves with $p$-rank $g-2$ or $g-3$ and $a$-number $a \geq 2$.

\mbox{}\\ \subsection{Sketch of the proof of Theorem \ref{TFVdG}} \label{SFVdGproof} \mbox{}\\

We give a sketch of the proof of Theorem \ref{TFVdG}; it uses the boundary of $\overline{\CM}_g$, whose points 
represent singular curves.

First, here is some more information about the $p$-rank of a singular curve.
Suppose $X$ has $2$ irreducible components $X_1$ and $X_2$, which are smooth curves 
which intersect in one ordinary double point.  
By \cite[Example 8, Page 246]{BLR}, ${\rm Jac}(X) \simeq {\rm Jac}(X_1) \oplus {\rm Jac}(X_1)$.
If $f_i$ is the $p$-rank of $X_i$, then the $p$-rank of $X$ is $f_1+f_2$.
It follows that the $p$-rank of a singular curve of compact type 
is the sum of the $p$-ranks of its components.

(More generally, the $p$-rank of a semi-abelian variety $A$ is 
$f={\rm dim}_{{\mathbb F}_p}{\rm Hom}(\mu_p, A)$.  The toric part of $A$ contributes to its $p$-rank.
This can be used to determine the $p$-rank of a curve of non-compact type, but we will not need that 
information in this paper.)

Thus, it is easy to construct a {\it singular} curve of genus $g$ with $p$-rank $f$, by constructing 
a chain of $f$ ordinary and $g-f$ supersingular elliptic curves, joined at ordinary double points.
This singular curve can be deformed to a smooth one, but it is not obvious that the $p$-rank stays constant 
in this deformation.
To prove that there is a {\it smooth} curve of genus $g$ with $p$-rank $f$, 
singular curves are still useful, but the argument must be done more carefully.

Let $\overline{{\mathcal M}}_{g;1}$ denote the moduli space whose points represent curves $X$ of genus $g$ 
together with a marked point $x$.  The dimension of $\overline{{\mathcal M}}_{g;1}$ is $3g-3+1$ for all $g \geq 1$. 

For each $1 \leq i \leq g-1$, there is a clutching morphism
\[\kappa_{i,g-i}: \overline{\CM}_{i;1} \times \overline{\CM}_{g-i;1} \to \overline{\CM}_g,\]
defined as follows:
if $\eta_1$ (resp.\ $\eta_2$) is a point representing a curve $X_1$ (resp.\ $X_2$) of genus 
$i$ (resp.\ $g-i$) with marked point $x_1$ (resp.\ $x_2$),
then $\kappa_{i,g-i}(\eta_1, \eta_2)$ represents the stable curve of genus $g$, with components 
$X_1$ and $X_2$, 
formed by identifying $x_1$ and $x_2$ in an ordinary double point. 
The image of $\kappa_{i,g-i}$ is the component $\Delta_i$ of the boundary of $\CM_g$.
The clutching morphism is a closed immersion if $i \not = g-i$ and is always a finite unramified map \cite[Corollary 3.9]{knudsen}.

\begin{proof} (Sketch of proof of Theorem \ref{TFVdG})
The proof is by induction on $g$.
When $g =2,3$, the result is true since the open Torelli locus is open and dense in $\CA_g$.
Suppose $g \geq 4$.

The dimension of $\overline{\CM}_g$ is $3g-3$.  There are singular curves that are ordinary, namely
chains of $g$ ordinary elliptic curves.  Since $\overline{\CM}_g$ is irreducible and 
the $p$-rank is lower semi-continuous, the generic geometric point of $\overline{\CM}_g$
is ordinary, with $p$-rank $g$.

Let $S$ be a component of $\overline{\CM}_g^f$.
The length of the chain which connects the ordinary Newton polygon $\nu_g$ to 
the largest Newton polygon having $(f,0)$ as a break point is $g-f$.
Using purity of the Newton polygon stratification \cite{oortpurity}, 
\[{\rm dim}(S) \geq (3g-3) -(g-f)=2g-3+f.\]

By \cite[Lemma 2.5]{fabervandergeer}, $S$ intersects $\Delta_i$ for each $1 \leq i \leq g-1$.
Since ${\rm codim}(\Delta_i, \overline{\CM}_g) = 1$, it follows from \cite[page 614]{V:stack} that 
${\rm dim}(S) \leq {\rm dim}(S \cap \Delta_i) +1$.

The $p$-rank of a singular curve of compact type is the sum of the $p$-ranks of its components, 
\cite[Example 8, Page 246]{BLR}.
As seen in \cite[Proposition 3.4]{AP:mono}, one can restrict the clutching morphism to the $p$-rank strata:
\[\kappa_{i,g-i}: \overline{\CM}^{f_1}_{i;1} \times \overline{\CM}^{f_2}_{g-i;1} \to \overline{\CM}^{f_1+f_2}_g.\]

This means that ${\rm dim}(S \cap \Delta_i)$ is bounded above by 
${\rm dim}(\overline{\CM}^{f_1}_{i;1}) + {\rm dim}(\overline{\CM}^{f_2}_{g-i;1})$, 
for some pair $(f_1,f_2)$ such that $f_1 + f_2 = f$.
Adding a marked point adds one to the dimension. 
By the inductive hypothesis (or an explicit computation when $i=1,g-1$), 
one checks that ${\rm dim}(\overline{\CM}^{f_1}_{i;1})=2i - 3 + f_1 +1$
and ${\rm dim}(\overline{\CM}^{f_2}_{g-i;1})=2(g-i) -3 +f_2+1$.
It follows that ${\rm dim}(S \cap \Delta_i) \leq 2g - 4 + f$. 
Thus ${\rm dim}(S) \leq 2g-3 +f$, which completes the proof.
\end{proof}

 \mbox{}\\ \section{An inductive result about Newton polygon and Ekedahl-Oort strata} \label{Sinductive} \mbox{}\\

This section contains the main result of the paper:
starting with a Newton polygon $\xi$ that can be realized for a smooth curve of genus $g$, 
the goal is to prove that any symmetric Newton polygon which is formed by adjoining slopes of $0$ and $1$ to $\xi$
can also be realized for a smooth curve (of larger genus and $p$-rank).
In the main result, I show this is possible under a geometric condition
on the stratum of ${\mathcal M}_g$ with Newton polygon $\xi$.

The importance of this result is that it allows us to restrict to the case of $p$-rank $0$ in
Question \ref{Qmot3}.  
This type of inductive process can be found in earlier work, 
e.g., \cite[Theorem 2.3]{fabervandergeer}, 
\cite[Section 3]{AP:mono}, \cite[Proposition 3.7]{Pr:large}, and \cite[Proposition 5.4]{APgen}.
However, Theorem \ref{Tinductive} is stronger than these results because 
it allows for more flexibility with the Newton polygon and Ekedahl-Oort type.

\mbox{}\\ \subsection{The main result} \label{Smainresult} \mbox{}\\

First, we fix some notation about Newton polygons and ${\rm BT}_1$ group schemes. 

\begin{notation} \label{Nstrata}
Let $\xi$ denote a symmetric Newton polygon (or a symmetric ${\rm BT}_1$ group scheme) occurring 
for principally polarized abelian varieties in dimension $g$. 
Let $\CA_g[\xi]$ be the stratum in $\CA_g$ whose geometric points represent principally polarized abelian varieties of 
dimension $g$ and type $\xi$.
Let $cd_\xi=  {\rm codim}({\mathcal A}_g[\xi], {\mathcal A}_g)$.
Let ${\mathcal M}_g[\xi]$ be the stratum in ${\mathcal M}_g$ whose geometric points represent smooth projective curves
of genus $g$ and type $\xi$.
\end{notation}

\begin{notation} \label{Nadde}
In the case that $\xi$ denotes a symmetric Newton polygon occurring in dimension $g$:
for $e \in \NN$, let $\xi^{+e}$ be the symmetric Newton polygon
in dimension $g+e$ such that
the difference between the multiplicity of the slope $\lambda$ in $\xi^{+e}$ 
and the multiplicity of the slope $\lambda$ in $\xi$ is $0$ if $\lambda \not \in \{0,1\}$ and is $e$ if 
$\lambda \in \{0,1\}$.  
\end{notation}

\begin{notation} \label{Nadde2}
In the case that $\xi$ denotes a symmetric ${\rm BT}_1$ group scheme occurring in dimension $g$:
for $e \in \NN$, let $\xi^{+e}$ be the symmetric ${\rm BT}_1$ group scheme
in dimension $g+e$ given by \[\xi^{+e}:= L^{e} \oplus \xi,\] where $L = \ZZ/p \oplus \mu_p$.
If $[\nu_1, \ldots, \nu_g]$ is the Ekedahl-Oort type of $\xi$, then 
$\xi^{+e}$ has Ekedahl-Oort type $[1,2, \ldots, e, \nu_1+e, \ldots, \nu_g +e]$. 
\end{notation}

\begin{theorem} \label{Tinductive}  
With notation as in \ref{Nstrata}, \ref{Nadde}, \ref{Nadde2},    
suppose that there exists an irreducible component $S=S_0$ of ${\mathcal M}_g[\xi]$ such that
${\rm codim}(S, {\mathcal M}_g)=cd_\xi$.
Then, for all $e \in \NN$, there exists a component $S_e$ of ${\mathcal M}_{g+e}[\xi^{+e}]$
such that ${\rm codim}(S_e, {\mathcal M}_{g+e})=cd_\xi$.
\end{theorem}

\begin{proof}
The proof is by induction on $e$, with the case $e=0$ being trivial.  By replacing $\xi$ by $\xi^{+(e-1)}$,
$S$ by $S_{e-1}$, and $g$ by $g+(e-1)$, it suffices to consider the case $e=1$.

As before, consider the clutching morphism 
\[\kappa_{1, g}:\overline{\mathcal M}_{1;1} \times \overline{\mathcal M}_{g;1} \to \overline{{\mathcal M}}_{g+1},\]
defined as follows:
if $\eta_1$ is a point representing an elliptic curve $E$ with origin $x_1$,
and $\eta_2$ is a point representing a curve $X_2$ of genus $g$ with marked point $x_2$, 
then $\kappa_{1,g}(\eta_1, \eta_2)$ represents the stable curve $X$ of genus $g+1$, 
with components $E$ and $X_2$, 
formed by identifying $x_1$ and $x_2$ in an ordinary double point.  By \cite[Example 8, Page 246]{BLR}, 
${\rm Jac}(X) \simeq {\rm Jac}(E) \oplus {\rm Jac}(X_2)$. The image of $\kappa_{1,g}$ 
is contained in the component $\Delta_1$ of the boundary of ${\mathcal M}_{g+1}$.

By hypothesis, ${\rm dim}(S) = {\rm dim}({\mathcal M}_g) - cd_\xi = 3g-3 - cd_\xi$.
Consider the forgetful morphism $\phi: {\mathcal M}_{g;1} \to {\mathcal M}_g$.
The fiber of $\phi$ over $\eta_2$ is isomorphic to $X$ and is thus irreducible.
Let $S' = \phi^*(S)$; it is irreducible and ${\rm dim}(S') = {\rm dim}(S) + 1$.
Let $Z= \kappa_{1,g-1}({\mathcal M}_{1;1} \times S')$ and note that ${\rm dim}(Z)={\rm dim}(S) + 2$.

The points of $Z$ represent singular curves of genus $g+1$ whose Newton polygon (resp.\ $p$-torsion group scheme) is
$\xi^{+1}$.
This shows that $\overline{{\mathcal M}}_{g+1}[\xi^{+1}]$ is non-empty.  
Let $S_1$ be a component of $\overline{{\mathcal M}}_{g+1}[\xi^{+1}]$ containing $Z$.
Since ${\rm codim}(\Delta_1, \overline{{\mathcal M}}_{g+1}) =1$, 
it follows from \cite[page 614]{V:stack} that ${\rm codim}(Z, S_1) \leq 1$.
So either (i) ${\rm dim}(S_1) = {\rm dim}(Z)$, in which case $Z$ is contained in the boundary   
of ${\mathcal M}_{g+1}$ or (ii) ${\rm dim}(S_1) = {\rm dim}(Z)+1$.

The next claim is that case (i) does not occur.
By definition, $cd_\xi=  {\rm codim}({\mathcal A}_g[\xi], {\mathcal A}_g)$.
By the dimension formulae for Newton polygon and Ekedahl-Oort strata found in Section \ref{Sdimstrata}, one can check that 
$cd_\xi = {\rm codim}({\mathcal A}_{g+1}[\xi^{+1}], {\mathcal A}_{g+1})$.
The generic point of $\overline{\CM}_{g+1}$ is ordinary so, 
by purity \cite{oortpurity}, ${\rm codim}(S_1, \overline{\CM}_{g+1}) \leq cd_\xi$.
So \[{\rm dim}(S_1) \geq 3(g+1)-3 - cd_\xi = {\rm dim}(S) + 3= {\rm dim}(Z) + 1.\]
This shows that only case (ii) can occur and ${\rm codim}(S_1, \overline{\CM}_{g+1})=cd_\xi$.

Furthermore, in case (ii), $S_1$ is not contained in $\Delta_1[\overline{{\mathcal M}}_{g+1}]$.  
Since the points of $S$ represent smooth curves, the generic point of $S_1$ 
is not contained in any other boundary component
and is thus in ${\mathcal M}_{g+1}$.  Thus the generic point of $S_1$ represents a smooth curve.
\end{proof}
  
\mbox{}\\ \subsection{Application}  \label{Sapplication} \mbox{}\\

Here is an application of Theorem \ref{Tinductive}.

Recall that $\xi_4$ is the symmetric Newton polygon in dimension $4$ having slopes $1/4$ and $3/4$.
Note that $\xi_4$ is the largest symmetric Newton polygon in dimension $4$ with $p$-rank $0$.
It follows that $cd_{\xi_4} = 4$.
For $e \in \NN$, recall that $\xi_4^{+e}$ is the symmetric Newton polygon in dimension $4+e$
such that the multiplicity of the 
slopes $0$ and $1$ is $e$ and the multiplicity of the slopes $1/4$ and $3/4$ is $1$.

If $A$ has Newton polygon type $\xi_4$, then its $p$-torsion group scheme 
is the unique symmetric ${\bt}_1$ group scheme $I_{4,1}$ as defined in Example \ref{Eir1}.
So the Ekedahl-Oort type of $\xi_4^{+e}$ is $L^e \oplus I_{4,1}$, where $L = \ZZ/p \oplus \mu_p$.

The following result was proven in \cite[Corollary 5.6]{APgen}.
Recall that $\CM_g^f$ denotes the $p$-rank $f$ stratum of $\CM_g$.
For most pairs $(g,f)$, it is not known whether $\CM_g^f$ is irreducible.

\begin{corollary} \label{Amoduli}
Let $p$ be prime.  If $g \geq 4$, then the Newton polygon $\xi_4^{+(g-4)}$ occurs at the generic point of at least one 
irreducible component of $\CM_g^{g-4}$.
\end{corollary}

\begin{proof}
The first claim is that, for all primes $p$, there exists a smooth curve defined over 
$\overline{\FF}_p$ with genus $4$ whose Newton polygon is $\xi_4$.
That follows from \cite[Lemma 5.3(a)]{APgen}, where the authors
use the Newton polygon stratification of the moduli space of principally polarized abelian fourfolds 
equipped with an action by $\ZZ[\zeta_3]$ of signature $(3, 1)$. 

Here is an alternative proof of the first claim.
Let $A$ be any $4$-dimensional principally polarized abelian variety over $\overline{\FF}_p$
with Newton polygon $\xi_4$. 
By \cite[Proposition 2.1]{shankartsimerman}, $A$ is isogenous to a Jacobian of a curve $X$ of compact type.
If $X$ were reducible, then ${\rm Jac}(X)$ would decompose as ${\rm Jac}(X_1) \oplus {\rm Jac}(X_2)$, which
would contradict the fact that 
$G_{1,3} \oplus G_{3,1}$ is indecomposable as a symmetric $p$-divisible group.
So $X$ is irreducible, and thus smooth.

By the first claim, $\CM_4[\xi_4]$ is non-empty.  Recall that $\xi_4$ is the largest Newton polygon 
in dimension $4$ with $p$-rank $0$.  Thus there is at least one component of $\CM_4^0$ whose 
generic geometric point has Newton polygon $\xi_4$.
Thus there is a component $S$ of ${\mathcal M}_4[\xi_4]$ such that
${\rm codim}(S, {\mathcal M}_4)={\rm codim}(\CM_4^0, \CM_4) = 4 = cd_\xi$.
The result follows from Theorem \ref{Tinductive}.
\end{proof}

Further applications of Theorem \ref{Tinductive} appear in \cite{LMPT}.

\mbox{}\\ \section{Open problems} \label{Sopen} \mbox{}\\

We end the paper with a few open problems about the $p$-rank $0$ stratum of ${\mathcal M}_g$.
The $p$-rank $0$ stratum of $\CA_g$ is irreducible for $g \geq 3$ by \cite{chaioort}.

\begin{question}
For $g \geq 4$, is the $p$-rank $0$ stratum of $\CM_g$ irreducible?
\end{question}

For the second question, note that the generic geometric point of $\CA_g^0$ represents an abelian variety whose Newton polygon has slopes $1/g$ and $(g-1)/g$
and whose $p$-torsion group scheme is isomorphic to $I_{g,1}$, the unique 
symmetric ${\rm BT}_1$ group scheme of rank $p^{2g}$, $p$-rank $0$, and $a$-number $1$ from Example \ref{Eir1}.

\begin{question} \label{Qopen1}
If $g \geq 5$, does there exist a smooth curve of genus $g$ whose Jacobian:\\
has Newton polygon with slopes $1/g$ and $(g-1)/g$?\\
has $p$-torsion group scheme isomorphic to $I_{g,1}$?
\end{question}

By \cite[Lemma 5.3(a)]{APgen} (see Corollary \ref{Amoduli}), 
the $g=4$ analogue of Question \ref{Qopen1} 
has a positive answer for all $p$.  The answer is also yes when $g = 5$ and $p \equiv 3,4,5,9 \bmod 11$;
and when $g=6$ and $p \equiv 2,4 \bmod 7$ \cite{LMPT}.

The final question is about the existence of a curve of genus $5$ with Newton polygon $\xi_5'$, 
where $\xi_5'$ denotes the symmetric Newton polygon with slopes $2/5$ and $3/5$.

\begin{remark} \label{R:g=5}
Let $A$ be a $5$-dimensional principally polarized abelian variety over $\overline{\FF}_p$
with Newton polygon $\xi'_5$. 
Let $T_5$ denote the intersection of $\CA_5$ with image of the Torelli morphism $T: \overline{\CM}_5 \to \tilde{\CA}_5$.
One can check that ${\rm dim}(\CA_5)= 15$ and ${\rm dim}(\overline{\CM}_5)=12$, so $T_5$ 
has codimension $3$ in $\CA_5$.  Let $x$ be the point of $\CA_5$ representing $A$.
By Example \ref{Ecidim}, the dimension of the isogeny leaf $I(x)$ at $x$ is $i(\xi)=4$.
Also $I(x)$ is proper by \cite[Proposition 4.11]{oortjams}. 
Since $I(x)$ has dimension $4$ and $T_5$ has codimension $3$, it is possible that $I(x)$ intersects $T_5$. 
If this intersection is non-trivial, 
$A$ would be isogenous to the Jacobian of a smooth curve $X$, which would have Newton polygon $\xi_5'$ by construction.
\end{remark}

\begin{question}
With notation as in Remark \ref{R:g=5}, 
does there exist a point $x$ of $\CA_5$, representing an abelian variety with slopes $2/5$ and $3/5$,
such that the isogeny leaf $I(x)$ intersects the Torelli locus $T_5$ in $\CA_5$?
\end{question}

%\bibliographystyle{amsalpha}
%\bibliography{oortchapter.bib}

\newcommand{\etalchar}[1]{$^{#1}$}
\providecommand{\bysame}{\leavevmode\hbox to3em{\hrulefill}\thinspace}
\providecommand{\MR}{\relax\ifhmode\unskip\space\fi MR }
\providecommand{\MRhref}[2]{%
  \href{http://www.ams.org/mathscinet-getitem?mr=#1}{#2}
}
\providecommand{\href}[2]{#2}

\end{document}